\documentclass[11pt,twoside,reqno]{amsart}
\allowdisplaybreaks
\usepackage{amsmath,amstext,amssymb,epsfig}
\usepackage{graphicx}
\usepackage{mathrsfs}
\calclayout
\makeatletter
\renewcommand{\@seccntformat}[1]{\bf\csname the#1\endcsname.}
\renewcommand{\section}{\@startsection{section}{1}
	\z@{.7\linespacing\@plus\linespacing}{.5\linespacing}
	{\normalfont\upshape\bfseries\centering}}
\renewcommand{\@biblabel}[1]{\@ifnotempty{#1}{#1.}}
\makeatother
\theoremstyle{plain}
\newtheorem{thm}{Theorem}[section]

\newtheorem{prop}[thm]{Proposition}

\theoremstyle{definition}
\newtheorem{ex}[thm]{Example}
\newtheorem{defn}[thm]{Definition}
\newtheorem{rem}{Remark}[section]

\usepackage{cancel}
\usepackage[parfill]{parskip}
\usepackage[german]{varioref}
\usepackage[all]{xy}
\usepackage{stmaryrd}
\usepackage{color}
\def\R{{\mathcal R}}
\def\T{{\mathcal A}}

\def\l{{\lamda}}

\def\T{{\mathcal T}}
\def\>{\succ}
\def\<{\prec}

\def\U{{\mathcal U}}

\def\l{\alpha}
\def\l{\lambda}
\def\p{\partial}

\def\m{\mu}

\begin{document}	
	\title[Sania Asif \textsuperscript{1}, Lamei Yuan\textsuperscript{2},   Yao Wang\textsuperscript{3*}]{ Cohomology of Twisted Rota-Baxter operators on Associative~Conformal Algebra}
	%%%%%%%%%%%%%%%%%%%%%%%%%%%%%%%%%%%%%%%%%%%%%%%%%%%%%%%%%%%%%%%%%%%%%%%%%%%%%%%%
	\author{Sania Asif \textsuperscript{1},  Lamei Yuan \textsuperscript{2}, Yao Wang \textsuperscript{3*}}
%%%%%%%%%%%%%%%%%%%%%%%%%%%%%%%%%	
   \address{\textsuperscript{1,3*}School of Mathematics and Statistics, Nanjing University of information science and Technology, Nanjing 210044, Jiangsu Province, PR China.}
   \address{\textsuperscript{2}School of Mathematics and Statistics, Harbin University of Technology, Harbin 150001, PR China.}
   	%\Tddress{\textsuperscript{3}School of Mathematics and Statistics, Nanjing University of information science and Technology, Nanjing 210044, Jiangsu Province, PR China.}
	
	%%%%
	%\lddress{\textsuperscript{3}Department of Mathematics, Zhejiang University, Hangzhou, Zhejiang Province,310027,PR China.}
	
	%\email{\textsuperscript{2}hongyanyong2008@yahoo.com}

	\email{\textsuperscript{1}11835037@zju.edu.cn, 200036@nuist.edu.cn}
	\email{\textsuperscript{2}lmyuan@hit.edu.cn}
	\email{\textsuperscript{3*}wangyao@nuist.edu.cn}
%	\email{\textsuperscript{3}wzx@zju.edu.cn}
	%%%%%%%%%%%%%%%%%%%%%%%%%%%%%%%%%%%%%%%%%%%%%%%%%%%%%%%%%%%%%%%%%%%%%%%%%%%%%%%
	%%%%%%%%%%%%%%%%%%%%%%%%%%%%%%%%%%%%%%%%%%%%%%%%%%%%%%%%%%%%%%%%%%%%%%%%%%%%%
	
	
	%[class=AMS]
	%%%%%%%%%%%%%%%%%%%%%%%%%%%%%%%%%%%%%%%%%%%%%%%%%%%%%%%%%%%%%%%%%%%%%%%%%%%%%
	%%%%%%%%%%%%%%%%%%%%%%%%%%%%%%%%%%%%%%%%%%%%%%%%%%%%%%%%%%%%%%%%%%%%%%%%%%
	\date{\today}
	\thanks{This work is supported by the Jiangsu Natural Science Foundation Project (Natural Science Foundation of Jiangsu Province), Relative Gorenstein cotorsion Homology Theory and Its Applications (No.BK20181406).}
	%%%%%%%%%%%%%%%%%%%%%%%%%%%%%%%%%%%%%%%%%%%%%%%%%%%%%%%%%%%%%%%%%%%%%%%%%
	%%%%%%%%%%%%%%%%%%%%%%%%%%%%%%%%%%%%%%%%%%%%%%%%%%%%%%%%%%%%%%%%%%%%%%%%%
	\begin{abstract} In this paper, we examine the concept of twisted Rota-Baxter (TRB) operators on associative conformal algebras. Our strategy begins by constructing an $L_\infty$-algebra using Maurer-Cartan elements derived from $H$-twisted Rota-Baxter ($H$-TRB) operators on associative conformal algebras. This structure leads us to explore the cohomology of the conformal $H$-TRB operator, which is characterized as the Hochschild cohomology of a specific associative conformal algebra with coefficients in a conformal bimodule. Furthermore, we study the linear and formal deformations of conformal $H$-TRB operators to explore the application of cohomology.
\end{abstract}\footnote{The third author is the corresponding author.}
\maketitle \section{Introduction}\label{introduction}%The associative and Lie conformal algebras structure theory was developed and generalized in \cite{8} and \cite{9}. Some features of the structure theory of conformal algebras (and their representations) of infinite type have been considered by various researchers in a series of works. In this field, one urgent problem is finding the conformal representation, cohomologies, and deformation of associative and Lie conformal algebras. Moreover, the structure theory of associative conformal algebras with a finite faithful representation of associative conformal bialgebra was developed in \cite{10} and \cite{11}, which can be viewed as a conformal analog of associative bialgebras \cite{12} and also as an associative analog of conformal bialgebras \cite{13}. G. Baxter first introduced the concept of Rota-Baxter operators in \cite{14} in its study about the fluctuation theory of probability. Following that, G.-C. Rota \cite{15} and P. Cartier \cite{16} studied the connection between Rota-Baxter operators and combinatorics. A Rota-Baxter operator is an inverse operator of an (invertible) derivation, and we call it an algebraic abstraction of the integral operator. In \cite{17}, Aguiar demonstrated how Rota-Baxter operators result in Loday's dendriform algebras. Rota-Baxter operators and associated structures have received much attention in the literature during the past twenty years(see \cite{18} for more information). Rota-Baxter operator deformations on Lie algebras have recently been developed in \cite{2}. The results have been generalized to associative algebras by the author in \cite{4}. Not only the Rota-Baxter operators and their generalizations on associative and Lie algebras are essential, but their significance also lies in conformal field theory. 
%%%%%%%%%%%%%%%%%%%%%%%%%%%%%%%%%%%%%%%%
 The theory of associative and Lie conformal algebra was evolved and generalized in \cite{8} and \cite{9}. Multiple researchers have considered certain perspectives of this class of algebras  and explored their representations theory. One immediate problem in this field is the determination of the conformal representation, cohomologies, and deformations of associative and Lie conformal algebras. Further, the structure theory of associative conformal algebras with a finite faithful representation of associative conformal bialgebra was established in \cite{11, 10}, which can be seen as a conformal analogue of associative bialgebras \cite{12} and also an associative analogue of conformal bialgebras \cite{13}. G. Baxter first proposed the concept of Rota-Baxter operators in \cite{14} while investigating probability fluctuation theory. Afterward,  P. Cartier \cite{16} and G.-C. Rota \cite{15} examined the connection between Rota-Baxter operators and combinatorics. As an algebraic abstraction, the Rota-Baxter operator serves as  an inverse operator of the (invertible) derivation. In \cite{17}, Aguiar illustrated how Rota-Baxter operators lead to Loday's dendriform algebras. Rota-Baxter operators and associated structures have acquired significant attention in the literature over the past two decades (for further information, see \cite{SYZ, 18}).
Recently,  Rota-Baxter operator's deformations on Lie algebras were explored in \cite{2}, and these results were generalized to associative algebra case in \cite{4}. The importance of Rota-Baxter operators and their generalizations in associative and Lie algebras surpasses their mathematical significance and finds applicability in conformal field theory.
\par %Recently, the detailed study on the cohomology of associative conformal algebra encouraged us to study conformal Rota-Baxter operators and their related structures. In \cite{3,20}, a generalization of a Rota-Baxter operator in the presence of a bimodule that is also known as relative Rota-Baxter operator or $\mathcal{O}$-operator was introduced, where more specific properties of $\mathcal{O}$-operators and Nijenhuis operators on associative conformal algebras were discussed. Let $\T$ be an associative conformal algebra and $\M$ be a conformal $\T$-bimodule. A $\mathbb{C}[\p]$-linear map $\R: \M \to \T$ is said to be a generalized Rota-Baxter operator if $\R$ satisfies $$\R (m)_\l \R (n)= \R(\R (m)_\l n + m_\l \R (n)),\textit{ for }m, n \in \M.$$ Generalized Poisson structures have an operator counterpart in Rota-Baxter operators. The above literature prompted us to propose the concept of conformal $H$-twisted Rota-Baxter operators as an operator analog of twisted Poisson structures (see \cite{21}). In this case, $H$ represents a Hochschild $2$-cocycle of $\T$ with coefficients in the conformal bimodule $\M$.
\par Lately, there has been substantial curiosity in exploring the cohomology of associative conformal algebras, which has prompted further examination of conformal Rota-Baxter operators and their associated structures. Previous studies, such as those referenced in \cite{3,20}, introduced the concept of relative Rota-Baxter operator or $\mathcal{O}$-operator which  generalizes the Rota-Baxter operator in the presence of a bimodule. These studies focused on exploring precise properties of $\mathcal{O}$-operators and Nijenhuis operators within the context of associative conformal algebras. Let $\T$ represent an associative conformal algebra and $\U$ a conformal $\T$-bimodule. A $\mathbb{C}[\p]$-linear map $\R: \U \to \T[\l]$ is defined as a conformal generalized Rota-Baxter operator or $\mathcal{O}$-operator if the following equation holds for all $u,v \in \U$: $$\R (u)_\l \R (v)= \R(\R (u)_\l v + u_\l \R (v)),$$ It is important to note that the notion of generalized Poisson structures is analogous to Rota-Baxter operators in terms of their operator counterparts. Extending on the earlier cited research and scholarly works, we propose the concept of conformal $H$-twisted Rota-Baxter operators $H$-TRB as an operator analog of twisted Poisson structure, as discussed in \cite{das,21}. In this scenario, the symbol $H$ denotes a Hochschild $2$-cocycle of $\T$ with coefficients in the conformal bimodule $\U.$
\par The objective of this research is to examine the cohomology and deformations of $H$-TRB operators on associative conformal algebras. Firstly, we established an $L_\infty$-algebra with the help of graded Lie algebra formed in \cite{4}, in which the Maurer-Cartan elements resemble $H$-TRB operators. This description permits the establishment of a cohomology theory for conformal $H$-TRB operator $\R$ on associative conformal algebras $\T$. Following that, we consider this cohomology to be the Hochschild cohomology of a particular algebra with coefficients in an appropriate conformal bimodule. Specifically, if  $\R:\U\to \T[\l]$ is a conformal $H$-TRB operator, we can conclude that $\U$ is an associative conformal algebra, see \cite{20}. We also illustrate the conformal $\U$-bimodule structure on $\T$. Additionally, we establish that the Hochschild cohomology and the cohomology of $R$ are isomorphic. We extend the  formal deformation theory of Gerstenhaber (see \cite{22}) to encompass the conformal $H$-TRB operator on $\T$. Within a deformation of a conformal $H$-TRB operator $\R$, We illustrate that the linear term in the cohomology of $R$ is a $1$-cocycle. This linear term is referred to as the "infinitesimal of the deformation". Additionally, we exhibit that equivalent deformations have cohomologous infinitesimals. We also establish a necessary condition for the rigidity of $\R$ in relation to the Nijenhuis elements and define Nijenhuis elements in relation to conformal $H$-TRB operators $\R$ .
\par %The structure of the paper is as follows. We review $H$-twisted Rota-Baxter operators on associative conformal algebras in Section $2$ and establish a few key terms to support our findings in the next sections. In Section $3$, we first provide the Maurer-Cartan characterization of a $H$-twisted Rota-Baxter operator $\R$, and then we define the cohomology of $\R$ using this characterization. Additionally, we consider the cohomology of $\R$ to be the Hochschild cohomology of a suitable algebra with coefficients in a suitable conformal bimodule. In Section $4$, deformations of $H$-twisted Rota-Baxter operators are discussed.
%%%%%%%%%%%%%%%%%%%%%%%%%%%%%%%%%
This paper is divided into several sections. In Section $2$, we provide a thorough review of $H$-twisted Rota-Baxter ($H$-TRB) operators on associative conformal algebras. This review serves as a basis and familiarizes important terminology that will be used in later sections. Section $3$ focuses on the Maurer-Cartan characterization of conformal $H$-TRB operators. We first describe this characterization and then explain the cohomology of $\R$ according to it. Additionally, we consider the cohomology of $\R$ to be the Hochschild cohomology of a suitable algebra with coefficients in a suitable conformal bimodule. In Section $4$, we discuss deformations of 
 conformal $H$-TRB operators in detail. We analyze various aspects and implications of these deformations within the given context.
\section{Preliminaries}
Consider an associative conformal algebra  $\T$  and a multiplication map on $\T$ denoted by $\tau_\l : \T\otimes \T \to \T[\l]$ where $\tau_\l(p,q) = p_\l q,$ for $p,q \in \T$. A vector space $\U$ together with the two conformal sesqui-linear maps $\mathfrak{l}_\l : \T \otimes \U \to\U[\l]$ and $\mathfrak{r}_\l: \U\otimes \T\to  \U[\l]$, defined by $(p, u) \mapsto p_\l u$  and $(u, p) \mapsto u _\l p$ respectively is called a conformal  $\T$-bimodule. These maps are referred to as left and right actions of $\l$ respectively, that satisfy the following identities:
\begin{enumerate}
	\item $(p_{\l}u)_{\l+\m}q= p_{\l} (u_{\m} q)$,
	\item $(u_{\l}p)_{\l+\m}q = u_{\l} (p_{\m} q)$,
	\item $(p_{\l}q)_{\l+\m} u= p_{\l} (q_{\m}u)$, 
\end{enumerate}for $p, q \in\T$ and $u\in \U$. 
\par The computation of Hochschild cohomology of an associative conformal algebra $\T$ with coefficients in the conformal $\T$-bimodule $\U$ involves the use of the cohomology of the cochain complex $(C^\circledcirc_{Hoch}(\T, \U), \delta_{Hoch})$. Here $$C^\circledcirc_{Hoch}(\T, \U)=\bigoplus_{n\geq 0}C^n_{Hoch}(\T, \U)$$  and  $$C^n_{Hoch}(\T, \U) = Hom(\T^{\otimes n}, \U),$$ for $n \geq 0$, is the space of $n$ cochains that encompass all multilinear maps of the following form\begin{equation*}
\begin{aligned}
f_{\l_1,\cdots,\l_{n-1}} :& \T^{\otimes n}\to \U[\l_1,\cdots,\l_{n-1}]\\
&p_1 \otimes\cdots\otimes p_n \mapsto f_{\l_1,\cdots,\l_{n-1}}
(p_{1},\cdots ,p_{n})
\end{aligned}
\end{equation*}Above defined map satisfies the conformal sesqui-linearity condition, given by
\begin{eqnarray*}f_{\l_1,\cdots,\l_{n-1}}
(p_1,\cdots,\p (p_j)
,\cdots,p_n) &= -\l_j f_{\l_1,\cdots,\l_{n-1}}
(p_{1},\cdots ,p_{n}),~~~~ j = 1,\cdots ,n-1, \\
f_{\l_1,\cdots,\l_{n-1}}(p_1,\cdots ,\p (p_n)) &= (\p +\l_1+\cdots+\l_{n-1})f_{\l_1,\cdots,\l_{n-1}}
(p_{1},\cdots,p_{n}).
\end{eqnarray*}
Additionally, the conformal Hochschild differential is given by the map $$\delta_{Hoch}:C^n_{Hoch}(\T, \U)\to C^{n+1}_{Hoch}(\T, \U),$$ defined as follows  \begin{equation}
\begin{aligned}
(\delta_{Hoch}f)_{\l_1, \l_2, \cdots, \l_n}(p_1, \cdots, p_{n+1}) &= {p_{1}}_{\l_1} f_{\l_2, \cdots, \l_n}(p_2, \cdots, p_{n+ 1})\\&+ \sum_{j= 1}^{n}(-1)^{j}f_{\l_1, \cdots, \l_j+ \l_{j+1}, \cdots, \l_n}(p_1, \cdots, p_{j- 1},{ p_{j}}_{\l_j} p_{j+1}, \cdots , p_{n+ 1}) \\&+ (-1)^{n+1}f_{\l_1,\cdots,\l_{n-1}}(p_{1},\cdots,p_{n})_{\l_1+ \l_2+\cdots+ \l_n} p_{n+1},
\end{aligned}
\end{equation}
for $f_{\l_1, \l_2, \cdots, \l_{n-1}}\in  C^n_{Hoch}(\T, \U)$.\par
Let us assume that $H_\l\in C^2_{Hoch}(\T, \U)$ that represents  a Hochschild $2$-cocycle. In other words $H_\l : \T^{\otimes2}\to \U[\l]$ fulfills the following equation for all $p_1, p_2,p_3  \in \T.$ 
$${p_1}_{\l_1} H_{\l_2}(p_2, p_3)- H_{\l_1+\l_2}({p_1}_{\l_1}p_2, p_3)+ H_{\l_1}(p_1, {p_2}_{\l_2}p_3)- H_{\l_1}(p_1 , p_2)_{\l_1+\l_2} p_3 = 0.$$
\begin{rem}  It should be noted that a $2$-cocycle $H$ can induce the structure of associative conformal algebra on the direct sum of vector spaces $\T\oplus \U$, which is referred to as the conformal $H$-twisted semi-direct product. The multiplication in this structure is described as follows:
$$(p, u)_{\l}^H (q, v) = (p_\l q, p _\l v + u_\l q + H_\l(p, q)),\text{ for }(p, u), (q, v) \in \T \oplus \U.$$
We use the notation $\T \ltimes_\l^{H} \U$ to represent the conformal $H$-twisted semi-direct product.
\end{rem}
\begin{defn}
A conformal $H$-twisted Rota-Baxter operator ($H$-TRB) is a  $\mathbb{C}[\p]$-linear map $\R : \U \to \T[\l]$, satisfying  $$\R (u)_\l \R(v) = \R(u _\l \R(v) + \R (u) _\l v + H_\l(\R (u), \R (v)))$$
for all $u,v\in  \U.$\end{defn} 
    The following result can be easily concluded.
\begin{prop}\label{prop5.2}
A conformal $H$-TRB operator on $\T$ is a $\mathbb{C}[\p]$-linear map $\R : \U \to \T[\l] $ such that the graph
of $\R$, defined by  $$Gr(\R) = \{(\R (u),u)|u \in \U\}$$
forms a subalgebra of the conformal $H$-twisted semi-direct product of $\T$ and $\U$ denoted by $\T\ltimes_\l^H \U$.
\end{prop}
\begin{proof}For any $(\R(u),u), (\R(v),v)\in Gr(\R)$, we have
$$(\R(u),u)_\l (\R(v),v) = (\R(u)_\l\R(v), \R(u)_\l  v+ u_\l \R(v)+H_\l(\R(u), \R(v))).$$	
Hence $\R$ is a conformal $H$-TRB operator iff\begin{equation*}
(\R(u)_\l\R(v), \R(u)_\l v+ u_\l \R(v)+H_\l(\R(u), \R(v)))\in Gr(\R)[\l].
\end{equation*} \end{proof}
\begin{ex} Any Rota-Baxter operator of weight $0$ on $\T$ is a conformal $H$-TRB operator with $H = 0$.\end{ex}
\begin{ex} 
    Consider an associative conformal algebra denoted by $\T$, a conformal $\T$-bimodule denoted by $\U$, and an invertible Hochschild $1$-cochain denoted by $h: \T \to \U$. In this context, it can be shown that the inverse of $h$, denoted by $\R = h^{-1}$, acts as an $H$-TRB operator. Where, the operator $H$ is given by $H = -\delta_{\text{Hoch}}h$, as described in Example 4.4 of \cite{3}.
\end{ex}
\begin{prop}\label{prop5.5}
Consider a $\mathbb{C}[\p]$-linear map $\R : \U \to \T[\l]$ as a conformal $H$-TRB operator. In this case, $\U$ has an associative conformal algebra structure with $\l$-product defined as follows:
$$u *_\l v = u_\l \R (v) + \R(u) _\l v + H_\l(\R (u), \R (v)), \textit{ for } u, v\in \U.$$
\end{prop}	By modifying the conformal $H$-TRB operator, we build a new conformal $H$-TRB operator with the Hochschild cocycle. We start with the following proposition:
\begin{prop} Assuming $\T$, $\U$  and $H$ are all associative conformal  algebra,  a conformal $\T$-bimodule and a Hochschild $2$-cocycle respectively.
We can establish that there exists an isomorphism between the conformal $H$-twisted semi-direct product $\mathcal{T} \ltimes_\ell^H \mathcal{U}$ and the conformal $H + \delta_h$-twisted semi-direct product $\mathcal{T} \ltimes_\l^{H + \delta_h} \mathcal{U}$, where $\delta_h$ is a Hochschild $1$-cochain.\end{prop}\begin{proof}For proof, readers are refer to \cite{3}(see Proposition 4.3).
\end{proof}Consider a conformal $H$-TRB operator  $\R :\U\to \T[\l]$  and  the subalgebra of the conformal $H$-twisted semi-direct product represented by $Gr(\R)\subset \T \ltimes_{\l}^{H} \U$. For any Hochschild $1$-cochain $h$, the set of elements $\varPhi_{h}(Gr(\R)) = \{(\R (u),u- h(\R (u)))|\quad u \in \U\}$ is a subalgebra within the algebra $\T \ltimes_\l^{H+\delta_h}\U$. It is important to mention that, $\varPhi_{h}(Gr(\R))$ may not necessarily correspond to the graph of $\R$. If the inverse of the map $(id- h\circ\R): \U \to \U$ exists, then $\varPhi_h(Gr(\R ))$ is the graph of the map $\R (id- h\circ \R)^{-1}$. In this context, $\R (id - h \circ \R)^{-1}$ is referred as  $(H + \delta_h)$-TRB operator.\\
Next, we present a perturbation to a conformal $H$-TRB operator $\R$ by applying a  Hochschild $1$-cocycle. Consider the graph $Gr(\R) \subset \T \ltimes_\l^H \U$, which represents a subalgebra of the twisted semi-direct product. We then explore the deformed subspace by using any Hochschild $1$-cocycle $h'$.
\begin{equation}
\xi_{h'}(Gr(\R )) := \{(\R (u),u + h'(\R (u)))|~ u\in \U \}
\end{equation}Then $\xi_h'(Gr(\R )) \subset \T \ltimes_\l^H \U$ is a subalgebra as
\begin{equation*}
\begin{aligned}
&(\R( u),u + h'(\R (u))) ._\l^H (\R (v),v + h'(\R (v)))
\\&= (\R(u)_\l\R (v),(\R (u))_\l v+ (\R (u)) _\l h'(\R (v)) + u _\l (\R (v)) +h'(\R (u))_\l \R (v) + H_\l(\R (u), \R (v)))
\\&= (\R(u)_\l\R (v),(\R (u))_\l v + u _\l (\R (v)) + h'(\R (u)_\l \R (v)) + H_\l(\R (u), \R (v)))\in  \xi_{h'}(Gr(\R)).
\end{aligned}
\end{equation*}
If $id+ h'\circ \R$  from  $\U$ to $\U$ is bijective, then the subset $\xi_{h'}(Gr(\R ))\subset \T\ltimes_\l^H \U$ represents the graph of the  map $\R (id + {h'}\circ \R)^{-1}$. In this context, $h'$ is referred to as $\R$-admissible $1$-cocycle. Therefore, the  map $\R (id + {h'}\circ\R )^{-1}$ satisfies the properties of a conformal $H$-TRB operator as stated in the  Proposition \ref{prop5.2}. This particular conformal $H$-TRB operator $\R$, is denoted by $\R_{h'}$.
\begin{prop}
	Suppose that we have a conformal $H$-TRB operator $\R$ and an $\R$-admissible $1$-cocycle $h'$. Then the derived  associative conformal algebra structures on $\U$,  from the conformal $H$-TRB operators $\R$ and $\R_{h'}$ are isomorphic.
\end{prop}
\begin{proof}Let us suppose  isomorphism of the vector space $id + h'\circ \R : \U \to \U.$ This implies that:
	\begin{equation*}
	\begin{aligned}
	&(id + h'\circ \R )(u)*_\l^{h'} (id + h'\circ \R)(v)
	\\&= (\R (u))_\l(id + h'\circ \R )(v) + (id + h' \circ \R )(u)_\l(\R (v)) + H_\l(\R (u), \R (v))
	\\&=\R (u) _\l v + u _\l \R (v) + (\R (u)) _\l h'(\R (v)) + h'(\R (u)) _\l (\R (v)) + H_\l(\R (u), \R (v))
	\\&	= u *_\l v + h'((\R (u))_\l(\R (v)))
	\\&= u *_\l v + h'\R (u *_\l v) = (id + h'\circ \R )(u *_\l v).
	\end{aligned}
	\end{equation*}
	It concludes the proof. 
\end{proof}
\section{Cohomology of $H$-TRB operator}%In this section, we include a ternary map on the underlying graded vector space of the graded Lie algebra to make it an $L_\infty$-algebra. This ternary bracket is made by using a Hochschild $2$-cocycle $H$. The Maurer-Cartan elements of this new $L_\infty$-algebra are precisely $H$-TRB operators. This characterization allows us to define cohomology for an $H$-TRB operator $\R$. Finally, we show that the cohomology of $\R$ can be seen as the Hochschild cohomology of $(\U, *_\l)$ with coefficients in a suitable conformal bimodule structure on $\T$.
%%%%%%%%%%%%%%%%%%%%%%%%%%%%%%%%%%%%%%%%
In this section, we introduce a ternary map to transform the graded Lie algebra into a $L_\infty$-algebra from its underlying graded vector space. In order to create the corresponding ternary bracket, a Hochschild $2$-cocycle $H$ is used. The newly formed $L_{\infty}$-algebra's Maurer-Cartan elements perfectly match the conformal $H$-TRB operator. The conformal $H$-TRB operator $R$'s cohomology can be described more easily thanks to this characterization. Finally, we show that the Hochschild cohomology of $(\U, *_\l)$ with coefficients in a suitable conformal bimodule structure on $T$ can be understood as the cohomology of $R$.

\subsection{Cohomology and Maurer-Cartan Characterization :}Consider that $\T$ is an associative conformal algebra and $\U$ be a conformal $\T$-bimodule . Let $\bigoplus_{n\geq 0}Hom(\U^{\otimes n}, \T)$ be the graded vector space  with the bracket operation described as  
\begin{equation}\label{eq28}
\llbracket A,B \rrbracket= (-1)^p[[m_c + l + r, A]_G, B]_G,
\end{equation} where $A \in Hom(\U^{\otimes a}, \T)$ and $B \in Hom(\U^{\otimes b}, \T)$. According to \cite{4}, it has been exhibited that this structure constitutes a graded Lie algebra. 
  In this case, the expression $m_c + l + r$ represents a Maurer-Cartan element, and $[\cdot,\cdot]_G$ represents the Gerstenhaber bracket on the direct sum vector space $\T \oplus \U$. The previous reference provides a clear illustration of $m_c + l + r$ and the bracket $\llbracket \cdot,\cdot \rrbracket$. Note that, for a $\mathbb{C}[\p]$-linear map $ \R: \U \to \T[\l] $ the following equation holds:
\begin{align}\label{eq29}
\llbracket \R, \R \rrbracket (u,v) = 2(\R(\R(u)_\l v+ u_\l \R(v))+{\R(u)}_\l\R(v))
\end{align} Here, the graded Lie bracket $\llbracket \cdot,\cdot \rrbracket$ is similar to the classical Schouten-Nijenhuis ($NS$) bracket of multi-vector fields on a manifold $\mathbb{M}$. In \cite{5}, the authors construct a ternary bracket on the space of multi-vector fields given a closed $3$-form on $\U$. This ternary bracket, together with the $NS$ bracket, forms an $L_\infty$-algebra structure on the graded space of multivector fields. This $L_\infty$- algebra's Maurer-Cartan elements look like twisted Poisson structures. This structure's associative correspondence can be formally expressed as follows.  Let $H\in C^{2}_{Hoch}(\T, \U)$ be a Hochschild $2$-cocycle of $\T$ with coefficients in $\U$. Define a ternary bracket $\llbracket \cdot, \cdot, \cdot \rrbracket$ of degree $-1$ on the graded vector space $\bigoplus_{m\geq 0}Hom(\U^{\otimes m}, \T)$ by
\small{\begin{equation} \begin{aligned}\label{eq300}&\llbracket A,B,C \rrbracket(u_1,\cdots, u_{a+b+c-1})\\&= (-1)^{abc}\{\sum_{1\leq j\leq a}(-1)^{(j-1)b} A_{\l_1,\cdots,\l_{j-1},{\l_j+\cdots+\l_{j+b+c-1}},\l_{j+b+c},\cdots,\l_{a+b+c-2}}\\&(u_1, \cdots, u_{j-1}, H_{\l_j+\cdots+\l_{j+b+c-2}}(B_{\l_j,\cdots,\l_{j+b-2}}(u_{j}, \cdots, u_{j+b-1}), C_{\l_{j+b},\cdots,\l_{j+b+c-2}}(u_{j+ b}, \cdots, u_{j+b+c-1})),\\& u_{j+b+c}, \cdots, u_{a+b+c-1}) \\&-(-1)^{bc}\sum_{1\leq j\leq a}(-1)^{(j-1)c} A_{\l_1,\cdots,\l_{j-1},{\l_j+\cdots+\l_{j+c+b-1}},\l_{j+c+b},\cdots,\l_{a+b+c-2}}\\&(u_1, \cdots, u_{j-1}, H_{\l_j+\cdots+\l_{j+c+b-2}}(C_{\l_j,\cdots,\l_{j+c-2}}(u_{j}, \cdots, u_{j+c-1}), B_{\l_{j+c},\cdots,\l_{j+c+b-2}}(u_{j+c}, \cdots, u_{j+c+b-1})),\\& u_{j+c+b}, \cdots, u_{a+c+b-1})\\&-(-1)^{ab}\sum_{1\leq j\leq b}(-1)^{(j-1)a} B_{\l_1,\cdots,\l_{j-1},{\l_j+\cdots+\l_{j+a+c-1}},\l_{j+a+c},\cdots,\l_{a+b+c-2}}\\&(u_1, \cdots, u_{j-1}, H_{\l_j+\cdots+\l_{j+a}+\cdots+\l_{j+b+c-2}}(A_{\l_j,\cdots,\l_{j+a-2}}(u_{j}, \cdots, u_{j+a-1}), C_{\l_{j+a},\cdots, \l_{j+a+c-2}}(u_{j+a}, \cdots, u_{j+a+c-1})),\\& u_{j+a+c}, \cdots, u_{a+b+c-1})\\&+(-1)^{a(b+c)}\sum_{1\leq j\leq b}(-1)^{(j-1)c} B_{\l_1,\cdots,\l_{j-1},{\l_j+\cdots+\l_{j+c+a-1}},\l_{j+c+a},\cdots,\l_{c+a+b-2}}\\&(u_1, \cdots, u_{j-1}, H_{\l_j+\cdots+\l_{j+c}+\cdots+\l_{j+c+a-2}}(C_{\l_j,\cdots,\l_{j+c-2}}(u_{j}, \cdots, u_{j+c-1}), A_{\l_{j+c},\cdots,\l_{j+c+a-2}}(u_{j+c}, \cdots, u_{j+c+a-1})),\\& u_{j+c+a}, \cdots, u_{c+a+b-1})\\&-(-1)^{ab+bc+ca}\sum_{1\leq j\leq c}(-1)^{(j-1)b} C_{\l_1,\cdots,\l_{j-1},\l_{j}+\cdots+\l_{j+b+a-1},\l_{j+b+a},\cdots,\l_{a+b+c-2}}\\&(u_1, \cdots, u_{j-1}, H_{\l_{j}+\cdots+\l_{j+b}+\cdots+\l_{j+b+a-2}}(B_{\l_j,\cdots,\l_{j+b-2}}(u_{j}, \cdots, u_{j+b-1}), A_{\l_{j+b},\cdots,\l_{j+b+a-2}}(u_{j+b}, \cdots, u_{j+b+a-1})),\\& u_{j+a+b}, \cdots, u_{a+b+c-1})\\&+(-1)^{c(a+b)}\sum_{1\leq j\leq c}(-1)^{(j-1)a} C_{\l_1,\cdots,\l_{j-1},\l_{j}+\cdots+\l_{j+a+b-1},\l_{j+a+b},\cdots,\l_{a+b+c-2}}\\&(u_1, \cdots, u_{j-1}, H_{\l_{j}+\cdots+\l_{j+a}+\cdots+\l_{j+a+b-2}}(A_{\l_j,\cdots,\l_{j+a-2}}(u_{j}, \cdots, u_{j+a-1}), B_{\l_{j+a},\cdots,\l_{j+a+b-2}}(u_{j+a}, \cdots, u_{j+a+b-1})),\\& u_{j+b+a}, \cdots, u_{a+b+c-1})\}\end{aligned} \end{equation}\small}
    It is obvious to follow that the ternary bracket  has graded skew symmetry. Following on from \cite{5}, it can be shown that the binary bracket  and the ternary bracket are incompatible in the framework of an $L_{\infty}$-algebra, with the trivial higher brackets. Furthermore, we have the following theorem.
\begin{thm} Consider that $\U$ is a conformal $\T$-bimodule and $H$ be a Hochschild $2$-cocycle. A $\mathbb{C}[\p]$-linear map $\R : \U \to \T[\l]$ is called a conformal $H$-TRB iff $\R \in Hom(\U, \T)$ is a Maurer-Cartan element in the $L_{\infty}$-algebra $(\bigoplus_{m\geq0} Hom(\U^{\otimes m}, \T), \llbracket \cdot,\cdot \rrbracket, \llbracket \cdot,\cdot , \cdot\rrbracket)$.
\end{thm}
\begin{proof} For a $\mathbb{C}[\p]$-linear map $\R : \U \to \T[\l]$, we have from Eq. (\ref{eq30}) that \begin{equation}\label{eq31}\llbracket \R,\R,\R \rrbracket(u_1, u_2) = -6\R (H_{\l_1}(\R (u_1), \R (u_2))),\textit{ for }u_1, u_2\in \U. \end{equation}Thus, from Eqs. (\ref{eq29}) and (\ref{eq31}), we get
\begin{equation*}
\begin{aligned}
\frac{1}{2}\llbracket \R,\R \rrbracket-\frac{1}{6}\llbracket \R,\R,\R \rrbracket& =\R(\R (u_1)_\l u_2+ {u_1}_\l \R (u_2)+ H_\l(\R (u_1), \R(u_{2})))- \R (u_1)_\l\R(u_2)
\end{aligned}
\end{equation*}
Hence, an element $\R \in Hom(\U, \T)$ is a Maurer-Cartan element iff it is a conformal $H$-TRB operator.
\end{proof}Now consider that  $\R$ is a conformal $H$-TRB operator. It is inferred from the above proposition that $\R$ generates a differential $d_\R : Hom(\U^{\otimes m}, \T) \to Hom(\U^{\otimes {m+1}}, \T),\textit{ for }m\geq 0$. This differential can be described as\begin{equation*}
d_\R (g) = \llbracket \R, g \rrbracket- \frac{1}{2}\llbracket \R,\R , g\rrbracket,\textit{ for } g \in Hom(\U^{\otimes m}, \T).
\end{equation*}
 From the definition of the brackets $ \llbracket \cdot,\cdot\rrbracket$ and $ \llbracket \cdot,\cdot , \cdot\rrbracket$, we can deduce that a linear map $g \in Hom(\U, \T)$ satisfies the following condition\\\begin{center}
	 $d_\R (g) =0$
\end{center}equivalently, 
\begin{equation}\label{eq334}
\begin{aligned}
\R(g(u_1)_\l u_2& + {u_1 }_\l g(u_2)) + g(\R (u_1) _\l u_2 + {u_1} _\l \R (u_2)) -\R (u_1)_\l g(u_2) \\&-g(u_1)_\l \R (u_2)+  \R(H_\l(g u_1, \R (u_2))+H_\l(\R (u_1), g u_2))+g(H_\l(\R (u_1), \R (u_2)))=0.
\end{aligned}
\end{equation}For each $n \geq 0,$ we can define
 the space of $m$-cocycles and $m$-coboundaries by $$Z^m_\R(\U, \T) = \{g \in Hom(\U^{\otimes m},\T)| d_\R (g) = 0\} $$ and $$B^m_\R(\U, \T) = \{d_\R (h)| h \in Hom(\U^{\otimes m-1},\T)\} $$ respectively. The cohomology groups corresponding to conformal $H$-TRB operator $\R$ is denoted by  $H^{m}_{\R}(\U, \T)$. note that, when given  $L_\infty$-algebra and a Maurer-Cartan element, a new $L_\infty$-algebra can be formed by using the Maurer-Cartan element see \cite{6,7}. Within the current scenario, it can be restated as follows.
 
\begin{thm}
  Let $\R : \U \to \T[\l]$ be a conformal $H$-TRB operator on $\T$. Then the graded vector space $\bigoplus_{m\geq 0}\mathrm{Hom}(\U^{\otimes m}, \T)$ provides a new $L_\infty$ algebra, known as the twisted $L_\infty$ algebra, with the following structure maps:
$$l_1(A) = d_\R (A),$$
$$ l_2(A,B) = \llbracket A,B \rrbracket- \llbracket A,B,C \rrbracket,$$
$$ l_3(A,B,C) = \llbracket A,B,C \rrbracket, $$
where $\llbracket \cdot,\cdot,\cdot \rrbracket$ denotes the Schouten bracket and $d_\R$ represents the differential derived from $\R$.
Additionally, for any linear map $\R: \U \to \T[\l]$, the sum $\R + \R'$ is a conformal $H$-twisted Rota-Baxter operator iff $\R'$ is a Maurer-Cartan element in the (new) twisted $L_\infty$ algebra.
\end{thm}
\begin{proof}
    The first step involves following the generic framework for constructing a twisted $L_\infty$-algebra, as detailed in \cite{6,7}. Moving on to the next part, we observe that
\begin{equation*}
\begin{aligned}
&\frac{1}{2}\llbracket \R+\R',\R+\R' \rrbracket -\frac{1}{6}\llbracket \R+\R',\R+\R',\R+\R' \rrbracket\\&=(\llbracket \R,\R' \rrbracket -\frac{1}{2}\llbracket \R,\R,\R' \rrbracket)+\frac{1}{2}(\llbracket \R',\R' \rrbracket-\llbracket \R,\R',\R' \rrbracket)-\frac{1}{6}\llbracket \R',\R',\R' \rrbracket\\&=l_1(\R')+\frac{1}{2}l_2(\R',\R')-\frac{1}{6}l_3(\R',\R',\R')\end{aligned}
\end{equation*}This indicates that the term $\R + \R'$
 serves as a conformal $H$-TRB operator, since it vanishes on the left-hand side. Likewise, $\R'$ can be determined as a Maurer-Cartan element within the twisted $L_\infty$-algebra. It completes the proof.
\end{proof}
\subsection{Hochschild cohomology of conformal $H$-TRB operator }
Consider that $\R: \U \to \T[\l]$ is a conformal $H$-TRB operator. Then according to Proposition \ref{prop5.5}, $\U$ has an associative conformal algebra structure $(\U,*_\l)$, which has a multiplication given as follows:
\begin{equation*}
u *_\l v = u _\l \R (v) + \R (u) _\l v + H_\l(\R( u), \R (v)),\textit{ for }u, v \in \U.
\end{equation*}
\begin{prop}
The maps $\mathfrak{l}^{\R}_{\l} : \U \otimes \T \to \T[\l]$ and $\mathfrak{r}^{\R}_{\l} : \T \otimes \U \to \T[\l]$ given by $$l^{\R}_{\l} (u, p) = \R (u)_\l p -\R (u_\l p + H_\l(\R (u), p))$$ and $$ r^{\R}_{\l} (p, u) = p_\l\R (u) - \R (p_\l u + H_\l(p, \R (u))),$$
for $ u \in \U, p\in \T,$ define a conformal $(\U, *_\l)$-bimodule on $\T$.
\end{prop}
\begin{proof}Consider that for any $p \in \T$ and $u, v\in \U$, we get
\begin{equation*}
\begin{aligned}
& \mathfrak{l}^{\R}_{\l}  (u *_\l v, p) - \mathfrak{l}^{\R}_{\l}  (u, \mathfrak{l}^{\R}_{\m} (v, p))\\&=\R (u *_\l v)_{\l+\m}p - \R ((u *_\l v) _{\l+\m} p) - \R H_{\l+\m}(\R (u *_\l v), p)\\& - \mathfrak{l}^{\R}_{\l}  (u, \R (v)_\m p - \R (v _\m a) - \R H_\m(\R (v), p))\\&=(\R(u)_\l\R(v))_{\l+\m}p - \R ((u _\l \R(v) +\R(u)_\l v+ H_\l(\R (u), \R (v))) _{\l+\m} p)\\& - \R H_{\l+\m}(\R (u _\l \R(v) +\R(u)_\l v+ H_\l(\R (u), \R (v))), p)\\& - \mathfrak{l}^{\R}_{\l} (u, \R (v)_\m p) + \mathfrak{l}^{\R}_{\l}(u, \R (v _\m p)) +\mathfrak{l}^{\R}_{\l}(u, \R H_\m(\R (v), p))\\&=(\R(u)_\l\R(v))_{\l+\m}p - \R ((u _\l \R(v)) _{\l+\m} p) -\R ((\R(u)_\l v) _{\l+\m} p)- \R((H_\l(\R (u), \R (v))) _{\l+\m} p)\\& - \R H_{\l+\m}(\R (u _\l \R(v)), p) - \R H_{\l+\m}(\R (\R(u)_\l v), p)- \R H_{\l+\m}(\R (H_\l(\R (u), \R (v))),p)\\& - \R(u)_\l(\R(v)_\m p)+\R(u_\l(\R(v)_\m p)+H_{\l}(\R(u),\R(v)_\m p))\\& + \R(u)_\l\R(v_\m p)-\R(u_\l(\R(v_\m p)))-\R H_{\l}(\R(u),\R(v_\m p))\\& + \R(u)_\l\R H_\m(\R (v), p)-\R(u_\l(\R H_\m(\R (v), p)))-\R H_{\l}(\R(u),\R H_\m(\R (v), p))\\&=\cancel{(\R(u)_\l\R(v))_{\l+\m}p} -\cancel{\R ((u _\l \R(v)) _{\l+\m} p)} -\cancel{\R ((\R(u)_\l v) _{\l+\m} p)}- \R((H_\l(\R (u), \R (v))) _{\l+\m} p)\\& - \R H_{\l+\m}(\R (u _\l \R(v)), p) - \R H_{\l+\m}(\R (\R(u)_\l v), p)- \R H_{\l+\m}(\R (H_\l(\R (u), \R (v))), p)\\& - \cancel{\R(u)_\l(\R(v)_\m p)}+\cancel{\R(u_\l(\R(v)_\m p))}+\R (H_{\l}(\R(u),\R(v)_\m p))\\& + \cancel{\R(u_\l \R(v_\m p))}+\cancel{\R(\R(u)_\l(v_\m p))}+\cancel{\R(H_{\l}(\R (u), \R(v_\m p)))}\\&-\cancel{\R(u_\l(\R(v_\m p)))}-\cancel{\R H_{\l}(\R(u),\R(v_\m p))}\\& + \R(\R(u)_\l H_\m(\R (v), p))+\cancel{\R(u_\l \R H_\m(\R (v), p))}+\cancel{\R(H_{\l}(\R(u), \R H_\m(\R (v), p)))}\\&-\cancel{\R(u_\l(\R H_\m(\R (v), p)))}-\cancel{\R H_{\l}(\R(u),\R H_\m(\R (v), p))}\\&=-\R H_{\l+\m}(\R (u )_\l \R(v), p)-\R((H_\l(\R (u), \R (v))) _{\l+\m} p)+\R (H_{\l}(\R(u),\R (v)_\m p))+\R(\R(u)_\l H_\m(\R (v), p))
\end{aligned} \end{equation*}Above equation is obtained by the cancellation of pair of terms, i.e, $\{1,8\},\{2,9\},\{3,12\},\{11,14\},\{13,15\},\{17,19\}$. Thus, we observe that the remaining terms are equal to \begin{equation*}
\begin{aligned}
\R(\delta_{Hoch}H_{\l,\m}(\R (u), \R (v), p))= 0.
\end{aligned}
\end{equation*} Similarly, we observe that
\begin{equation*}
\begin{aligned}
&\mathfrak{r}^{\R}_{\l} (\mathfrak{l}^{\R}_{\l}(u, p), v) - \mathfrak{l}^{\R}_{\l}(u, \mathfrak{r}^{\R}_{\m} (p, v))\\
&=\mathfrak{r}^{\R}_{\m} (\R (u)_\l p -\R (u_\l p) -\R H_\l(\R (u), p), v) - \mathfrak{l}^{\R}_{\l}(u, p_\m\R (v) -\R(p_\m v) -\R H_\m(p, \R (v)))\\
&= (\R(u)_\l p)_{\l+\m}\R (v) - \R (u _\l p)_{\l+\m} \R(v) - \R H_\l(\R (u), p)_{\l+\m}\R (v) -\R ((\R (u)_\l p)_{\l+\m} v) \\
&+ \R (\R (u _\l p)_{\l+\m} v)+ \R (\R H_{\l}(\R (u), p)_{\l+\m} v) - \R H_{\l+\m}(\R (u)_\l p, \R (v)) + \R H_{\l+\m}(\R (u _\l p), \R (v)) \\
&+ \R H_{\l+\m}(\R H_\l(\R (u), p), \R (v))- \R (u)_\l (p_\m \R (v)) + \R(u)_\l \R (p _\m v) + \R (u)_\l \R H_{\m}(p, \R (v)) \\
&+ \R (u _\l (p_\m \R (v))) - \R (u _\l \R (p _\m v))- \R (u _\l \R H_\m(p, \R (v)))\\
&+ \R H_{\l}(\R (u), p_\m\R (v)) - \R H_{\l}(\R (u), \R (p _\m v)) - \R H_{\l}(\R( u), \R H_\m(p, \R (v)))\\
&=-\R ((u _\l p) _{\l+\m} \R (v)) - \R (\R (u _\l p) _{\l+\m} v) - \R H_{\l+\m}(\R (u_\l p), \R (v)) - \R (H_{\l}(\R (u), p) _{\l+\m} \R (v))\\
&- \R (\R H_{\l+\m}(\R (u), p)_{\l+\m} v)-\R H_{\l+\m}(\R H_{\l}(\R (u), p), \R (v)) - \R (((\R (u))_\l p) _{\l+\m} v) + \R (\R (u_\l p)_{\l+\m}v) \\&+ \R (\R H_\l(\R (u), p)_{\l+\m} v)-\R H_{\l+\m}(\R (u)_\l p, \R (v))+ \R H_{\l+\m}(\R (u _\l p), \R (v)) + \R H_{\l+\m}(\R H_\l(\R (u), p), \R (v)) \\
& + \R (u_\l \R (p_\m v)) + \R (\R (u)_\l (p _\m v)) + \R H_{\l}(\R (u), \R(p _\m v))+ \R (u _\l \R H_\m(p, \R (v)))+ \R (\R (u) _\l H_\m(p, \R (v))) \\&+ \R H_{\l}(\R (u), \R H_\m(p, \R (v))) + \R (u _\l (p_\m\R (v))) -\R (u_\l \R (p_\m v))- \R (u _\l \R H_\m(p, \R (v))) \\&+ \R H_{\l}(\R (u), p_\m\R (v)) -\R H_{\l}(\R (u), \R (p _\m v)) -\R H_{\l}(\R (u), \R H_\m(p, \R (v)))
\\&\overset{cancellation}{=}-\R(H_\l(\R (u), p)_{\l+\m}\R (v)) - \R H_{\l+\m}(\R (u)_\l p, \R (v)) + \R (\R (u) _\l H_\m(p, \R (v))) + \R H_{\l}(\R (u), p_\m\R (v))\\&
= \R (\delta_{Hoch}H_{\l,\m}(\R (u), p, \R (v))) =0 
\end{aligned}
\end{equation*}
and 
\begin{equation*}
\begin{aligned}
&\mathfrak{r}^{\R}_{\m} (\mathfrak{r}^{\R}_{\l} (p, u), v) - \mathfrak{r}^{\R}_{\l} (p, u *_\m v)
\\&= \mathfrak{r}^{\R}_{\m} (p_\l\R (u) -\R (p _\l u) - \R H_\l(p,\R (u)), v) - p_\l\R(u *_\m v) + \R (p _\l (u *_\m v)) + \R H_{\l}(p, \R (u *_\m v))
\\&= (p_\l\R (u))_{\l+\m}\R (v) -\R(p_\l u)_{\l+\m}\R (v) - \R H_\l(p, \R (u))_{\l+\m}\R (v)-\R ((p_\l\R (u)) _{\l+\m} v) + \R (\R (p _\l u) _{\l+\m} v) \\&+ \R (\R H_{\l}(p, \R( u))_{\l+\m} v)
-\R H_{\l+\m}(p_\l\R (u), \R (v)) + \R H_{\l+\m}(\R (p_\l u), \R (v)) + \R H_{\l+\m}(\R H_{\l}(p, \R (u)), \R (v)) \\&
-(p_\l\R(u))_{\l+\m}\R (v) + \R (p _\l (u_\m \R (v))) + \R (p _\l (\R(u) _\m v)) + \R (p _\l H_\m(\R (u), \R (v))) + \R H_{\l}(p, \R (u)_\m\R(v))\\&
= -\R ((p _\l u) _{\l+\m}\R (v)) -\R(\R (p _\l u) _{\l+\m} v) -\R H_{\l+\m}(\R (p _\l u), \R (v)) -\R (H_\l(p, \R (u)) _{\l+\m} \R (v))\\& - \R(\R H_{\l}(p, \R (u))_{\l+\m} v)
-\R H_{\l+\m}(\R H_{\l}(p, \R (u)),\R (v)) - \R ((p_{\l}\R (u)) _{\l+\m}v) + \R (\R (p _{\l} u) _{\l+\m} v)  \\& + \R (\R H_\l(p, \R (u))_{\l+\m} v)- \R H_{\l+\m}(p_{\l}\R (u), \R (v))+ \R H_{\l+\m}(\R (p _\l u), \R (v)) + \R H_{\l+\m}(\R H_\l(p, \R (u)), \R (v))  \\&+ \R (p _\l(u _\m \R (v)))+\R (p _\l (\R (u)_\m v))+\R(p_\l H_\m (\R (u),\R (v)))+\R H_\l(p, \R (u)_\m\R (v))\\&
\overset{cancellation}{=}- \R (H_\l(p, \R (u))_{\l+\m}\R (v))  - \R H_{\l+\m}(p_\l\R (u), \R (v)) + \R (p _\l H_\m(\R (u), \R (v))) + \R H_{\l}(p,\R(u)_\m\R (v))\\&
= \R (\delta_{Hoch}H_{\l,\m}(p,\R (u), \R (v))) = 0.\end{aligned}
\end{equation*}This shows that $\mathfrak{l}^{\R}_{\l}$, $\mathfrak{r}^{\R}_{\l}$ defines a conformal $(\U, *_\l)$-bimodule structure on $\T$.
\end{proof}
The aforementioned result shows that it is possible to compute the Hochschild cohomology of the associative conformal algebra $(\U, *_\l)$ with coefficients in the conformal $\U$-bimodule $\T$. This involves studying the cochain complex $\{C_{Hoch}^{\circledcirc}(\mathcal{U}, \mathcal{T}), \delta_{Hoch}\}$, where $C^m_{Hoch}(\mathcal{U}, \mathcal{T})$ denotes the space of $m$-cochains that are linear maps from $\mathcal{U}^{\otimes m}$ to $\T$, and the coboundary map $\delta_{Hoch}: C^m_{Hoch}(\U, \T) \to C^{m+1}_{Hoch}(\U, \T)$ is provided by:
\begin{eqnarray}\label{eq40}
\begin{aligned}
&(\delta_{Hoch}(g))(u_1,\cdots, u_{m+1})\\&=\R (u_1)_{\l_1}g_{\l_2,\cdots,\l_m}(u_2,\cdots, u_{m+1})- \R ({u_1}_{\l_1} g_{\l_2,\cdots, \l_m}(u_2, \cdots , u_{m+1})) \\&- \R H_{\l_1+\cdots+\l_m}(\R (u_1), g_{\l_2,..\l_m}(u_2, \cdots , u_{m+1})) +
\sum_{j=1}^{m} (-1)^j g_{\l_1,...,\l_{j-1},\l_j+\l_{j+1},...,\l_m}(u_{1}, \cdots, u_{j-1},{u_j}_{\l_j}\R(u_{j+1}) \\&+ \R (u_j)_{\l_j }u_{j+1} + H_{\l_j}(\R (u_{j}), \R (u_{j+1})), u_{j+1},\cdots, u_{m+1})\\&
+(-1)^{m+1}g_{\l_1,\cdots,\l_{m-1}}(u_1,\cdots, u_m)_{\l_1+\cdots+\l_{m}}\R (u_{m+1})-(-1)^{m+1}\R (g_{\l_1,\cdots,\l_{m-1}}(u_1, \cdots, u_m) _{\l_1+\cdots+\l_m} u_{m+1})\\& - (-1)^{m+1}\R H_{\l_1+\cdots+\l_m}(g_{\l_1,\cdots,\l_{m-1}}(u_1, \cdots , u_m), \R (u_{m+1}))
\end{aligned}\end{eqnarray} for $g \in C^{m}_{Hoch}(\U, \T)$. The Hochschild cohomology group correspondng to the cochain complex $\{C_{Hoch}^{\circledcirc}(\U, \T), \delta_{Hoch}\}$ is denoted by $H^{\circledcirc}_{Hoch}(\U, \T)$. Thus, we get
\begin{equation*}
    H^0_{Hoch}(\U, \T) = \{p \in \T| p_\l \R (u) -\R (u)_\l p + \R H_\l(\R (u), p) - \R H_\l(p, \R (u)) = \R (p _\l u - u_{\l} p), ~~\forall~~~ u\in \U.
\end{equation*}
It follows from Eq. (\ref{eq40}) that a $\mathbb{C}[\p]$-linear map $g : \U \to \T[\l]$ is a Hochschild $1$-cocycle iff $g$ fulfills
\begin{equation*}
\begin{aligned}
\R(u)_\l g(v) + g(u)_\l\R (v) &= \R (u _\l g(v) + H_\l(\R (u), g(v))+g(u)_\l v + H_\l(g(u), \R (v)))
\\&+ g(u _\l \R (v) + \R (u) _\l v + H_\l(\R (u), \R (v))), 
\end{aligned}
\end{equation*}for all $u, v \in \U.$The cocycle condition has the same meaning as the cocycle condition described in Eq. (ref. eq334). In fact, we can additionally generalize this result as follows.
\begin{prop} Let $\R : \U \to \T[\l]$ be a conformal $H$-TRB operator. Then we have
$$d_{\R} (g) = (-1)^m \delta_{Hoch}(g)$$ for any $g \in Hom(\U^{\otimes m}, \T).$ 
\end{prop}
\begin{proof}The detailed illustration of the bracket in Eq. (\ref{eq28}) results into the following expression
\begin{equation*}
\begin{aligned}
&\llbracket\R,g\rrbracket \\&= (-1)^m \R (u_{1})_{\l_1}g_{\l_2,...\l_m}(u_2, \cdots, u_{m+1})-\R ({u_1}_{\l_1}g_{\l_2,...\l_m}(u_2, \cdots , u_{m+1})) \\&+
\sum_{i=1}^{m} (-1)^j g_{\l_1,...\l_m}(u_1,\cdots, u_{j-1}, {u_j}_{\l_j} \R (u_{j+1}) + \R(u_j) _{\l_j} u_{j+1}, u_{j+1}, \cdots, u_{m+1})\\&+ (-1)^{m+1}g_{\l_1,...\l_{m-1}}(u_{1}, \cdots, u_m)_{\l_1+...+\l_m}\R(u_{m+1})- (-1)^{m+1}\R (g_{\l_1,...,\l_{m-1}}(u_1,\cdots, u_{m})_{\l_1+...+\l_m} u_{m+1}).
\end{aligned}\end{equation*}Additionally, from Eq. (\ref{eq300})  we have
\begin{equation*}
\begin{aligned}\llbracket\R,\R,g\rrbracket &= - 2(-1)^m (\R H_{\l_1+\cdots+\l_m}(\R (u_1), g_{\l_2,\cdots,\l_{m-1}}(u_2, \cdots , u_{m+1})) \\&+ (-1)^{m}\R H_{\l_1+\cdots+\l_m}(g_{\l_1,\cdots,\l_{m-1}}(u_1,\cdots , u_m), \R (u_{m+1}))\\&
+\sum_{ j=1}^{m} (-1)^{j}g_{\l_1,\cdots,\l_{j-1},\l_{j}+\l_{j+1},\cdots,\l_{m-1}}(u_1, \cdots, u_{j-1}, H_{\l_j}(\R (u_j), \R (u_{j+1})), u_{j+1}, \cdots , u_{m+1})).
\end{aligned}
\end{equation*}
Finally
\begin{align*}d_{\R}(g) = \llbracket \R , g \rrbracket - \frac{1}{2} \llbracket \R ,\R , g \rrbracket = (-1)^{m}\delta_{Hoch}(g).
\end{align*}\end{proof}As a result of the aforementioned proposition, we have the following theorem.
\begin{thm}
Consider an associative conformal algebra $(\U, *_\l)$ and a conformal $H$-TRB operator $\R: \U\to \T[\l]$ on it. In this case, the cohomology $H^\circledcirc_\R(\U, \T) \simeq H^{\circledcirc}_{Hoch}(\U, \T)$. Where $H^{\circledcirc}_{Hoch}(\U, \T)$ denotes the Hochschild cohomology of $(\U, *_\l)$ with coefficients in the conformal $\U$-bimodule $\T$.
\end{thm}
\section{Deformations of $H$-TRB operator on associative conformal algebra}
In the current section, we focus on the linear and formal deformations of the conformal $H$-TRB operator within the associative conformal algebra context. To achieve this, we employ the use of a $H$-TRB (that arises from trivial linear deformations as a means to introduce Nijenhuis elements associated with it).
Additionally, we are able to establish a sufficient condition that proves the rigidity of $H$-TRB on associative conformal algebra through the analysis of these Nijenhuis elements.
\subsection{Linear deformations.}Assume  that  $\T$ is an associative conformal algebra, $\U$  is a conformal $\T$-bimodule  and  $H$ be a Hochschild $2$-cocycle.  Further, suppose that $\R$ is an $H$-TRB operator on an associative conformal algebra $\T$. An $H$-TRB operator can be linearly deformed by adding the parameterized sum $\R_t = \R + t\R_1$, for some $\mathbb{C}[\p]$-linear map $\R_1 \in Hom(\U, \T)$. This means that for all values of $t$, $\R_t$ remains an $H$-TRB operator.     In the context of linear deformation, we represent $\R_1$ as the generator. Hence, for the linear deformation equation given  by $\R_t = \R + t\R_1$, the following condition must be satisfied for $u, v \in \U$ and for all $t$
\begin{equation}
\R_t(u)_\l\R_t(v) = \R_t(u _\l \R_t(v) + \R_t(u)_\l v + H_\l(\R_t(u), \R_t(v)))
\end{equation} By correlating the coefficients of different powers of $t$, we have the following
\begin{equation}
\R(u)_\l\R(v) + \R(u)_\l\R(v) = \R(u_\l \R(v) + \R(u)_\l v + H_\l(\R (u), \R (v)))
\end{equation}
\begin{equation}\label{eq30}
\begin{aligned}
\R(u)_\l\R_1(v) + \R_1(u)_\l\R (v) =& \R(u_\l \R_1(v) + \R_1(u)_\l v + H_\l(\R_1 u, \R (v))+H_\l(\R (u), \R_1 (v)))\\&+ \R_1(u_\l \R(v) + \R(u)_\l v + H_\l(\R (u), \R (v))) 
\end{aligned}
\end{equation}
\begin{equation}
\begin{aligned}
\R_1(u)_\l\R_1(v)=&\R_1(u_\l \R_1(v)+\R_1(u)_\l v+ H_\l(\R_1(u), \R(v))+H_\l(\R(u), \R_1(v)) )\\&+\R H_\l(\R_1(u), \R_1(v))
\end{aligned}
\end{equation}
\begin{equation}
\R_1(H_\l(\R_1(u), \R_1(v)))= 0.
\end{equation}Observe that the  Eq. (\ref{eq30}) is equivalent to the fact that $\R_1$ is a $1$-cocycle in the cohomology of $\R$.
\begin{defn}
	A conformal $H$-TRB operators morphism  from $\R$ to $\R'$ on associative conformal algebra $\T$ is a pair $(\phi, \psi)$ consisting of  an algebra morphism $\phi: \T\to \T'$ and a linear map $\psi : \U \to \U$ that satisfies the following equations
	\begin{eqnarray}
	&\label{eq33}\psi(p _\l u) = \phi(p) _\l \psi(u)\textit{ and } \psi(u _{\l} p) = \psi(u)_{\l} \phi(p),\\&
	\label{eq333}
	\psi \circ H_\l = H'_\l\circ (\phi \otimes \phi),\\&\label{eq3333}
	\phi \circ \R = \R'\circ \psi,
	\end{eqnarray}for $p \in \T$, $u \in \U$ and $\l\in \mathbb{C}$ .
	\end{defn}
	\begin{defn}
	Two linear deformations 
	of a conformal $H$-TRB operator $\R$,  $\R_t = \R+ t\R_1 $ and $\R'_t = \R + t\R_1$ defined on associative conformal algebra are equivalent if there exists $p\in \T$ such that the pair $(\phi_t, \psi_t)$ given by
	\begin{equation*}
	(\phi_t= id_\l+t(ad_p^l- ad_p^r), \psi_t= id_\U+ t(\mathfrak{l}_p- \mathfrak{r}_p+ H_\l(p, \R-)- H_\l(\R-, p)))
	\end{equation*}is a conformal  $H$-TRB operators  morphism from $\R_t\to \R'_t$.
	\end{defn}
	Thus, an algebra morphism $\phi_t : \T \to \T$  implies that
	\begin{equation}\label{eq34}
	(p_\l q - q_{-\p-\l}p)_\m(p_\l r - r_{-\p-\l}p) = 0,\text{ for }q, r \in \T.
	\end{equation}
	The  Eq. (\ref{eq33}) that defines the morphism of conformal $H$-TRB operator can be equivalently written as 
	\begin{equation}\label{eq35}\left\{ 
	\begin{array}{ c l }
	&p_\l (q_\m u) - (q _\m u)_{-\p-\l}p + H_{\l}(p, \R (q_\m u))- H_{\l+\m}(\R (q _\m u), p)  \\&=(p_\l q - q_{-\p-\l}p)_{\l+\m} u + q _\m(p _\l u - u _{-\p-\l} p + H_\l(p, \R (u)) - H_\l(\R (u), p)), \\&
	(p_\l q - q_{-\p-\l}p) _\m (p _\l u - u_{-\p-\l} p + H_{\l}(p, \R (u)) - H_{\l}(\R (u), p)) = 0   \quad \text{ for } q \in \T,~~u \in \U.
	\end{array}
	\right.\end{equation}
	%\begin{equation}\label{eq35}\left\{ \begin{array}{ c l } &p_\l (b_\m u) - (b _\m u)_{-\p-\l}p + H_{\l}(p, \R (b_\m u))-  \\&H_{\l+\m}(\R (b _\m u), p)  =(p_\l b - b_{-\p-\l}p)_{\l+\m} u + b _\m(p _\l u - u _{-\p-\l} p + H_\l(p, \R (u)) - H_\l(\R (u), p)),\\&(p_\l b - b_{-\p-\l}p) _\m (p _\l u - u_{-\p-\l} p + H_{\l}(p, \R (u)) - H_{\l}(\R (u), p)) = 0, \end{array} \right. \end{equation} 
	Similarly, the Eqs. (\ref{eq333}) and (\ref{eq3333}) are respectively equivalent to the following
	\begin{equation}\label{eq36}
	\left\{ \begin{array}{ c l }
	p _\l H_\m(q, r) - H_\m(q, r) _{-\p-\l} p + \\H_{\l}(p, \R H_\m(q, r)) -H_{\l+\m}(\R H_\m(q, r), p) &= H_{\l+\m}(p_\l q-q_{-\p-\l}p, r) + H_{\m}(q, p_\l r - r_{-\p-\l}p),\\ H_{\l+\m}(p_\l q - q_{-\p-\l}p, p_\l r - r_{-\p-\l}p) &= 0,\quad { for } ~~q, r \in \T.
	\end{array} \right.
	\end{equation}
	\begin{equation}\label{eq37}
	\left\{ \begin{array}{ c l }
	\R_1(u) + p_\l \R(u) - \R(u)_{-\p-\l}p &= \R (p_\l u - u _{-\p-\l} p + H_\l(p, \R (u)) - H_\l (\R (u), p)) + \R'_1(u),\\
	p_\l\R_1(u) - \R_1(u)_{-\p-\l}p &= \R'_1(p _\l u - u _{-\p-\l}p + H_\l(p, \R (u)) - H_\l(\R (u), p))\end{array} \right. \end{equation}
	It is noted that  the first identity in Eq. (\ref{eq37}) suggests that  for $u\in \U$, we have $$\R_1(u)- \R'_1(u) = d_\R(p)_\l(u),$$  Thus, we have the following result.
	\begin{thm}
	Consider $\R_t = \R + t\R_1$ represent a linear deformation of $H$-TRB operator $\R$. In this scenario, $\R_1$ becomes a $1$-cocycle within the cohomology of $\R$. The  class of $\R_1$ solely relies on the equivalence class of the linear deformation $\R_t$.
	\end{thm}
	Building on the proceeding discussions, we define the concept of the Nijenhuis element linked with the  action of $H$-TRB operator on an associative conformal algebra $\T$.
	\begin{defn}
		An element $p\in \T$ is said to be  a Nijenhuis element associated with the $H$-TRB operator $\R$ if $ p $ satisfies the  following equation:
		\begin{equation}
		p_\m(\mathfrak{l}^\R_{\l} (u, p) - \mathfrak{r}^\R_{\l} (p, u)) -(\mathfrak{l}^\R_{\l} (u, p) - \mathfrak{r}^\R_{\l}(p, u))_{-\p-\m}p = 0
		\end{equation}In addition, Eqs. (\ref{eq34}), (\ref{eq35}), (\ref{eq36}), and (\ref{eq37})  must also hold.
	\end{defn}The set of all Nijenhuis elements associated with $\R$ is represented by $Nij(\R)$.  Based on the previous discussions, we infer that a Nijenhuis element $p$ can arise from the trivial linear deformation of $\R$.  In the subsequent subsection, we evaluate the rigidity of a conformal $H$-TRB operator and established the prerequisite for the rigidity, that can be described in relation to the Nijenhuis element.
	\subsection{Formal deformations} In the present subsection, we introduce the formal deformation theory to encompass conformal $H$-TRB operators. Suppose an associative conformal algebra $\T$, a conformal $\T$-bimodule $\U$, and a Hochschild $2$-cocycle $H$. It is important to note that the associative multiplication on $\T$ leads to an associative multiplication on $\T\llbracket t \rrbracket$, which represents the set of formal power series in $t$ with coefficients taken  from $\T$. Furthermore, the conformal $\T$-bimodule structure on $\U$ results in a conformal $\T\llbracket t \rrbracket$-bimodule structure on $\U\llbracket t \rrbracket$. In addition, the Hochschild $2$-cocycle $H$ yields a Hochschild $2$-cocycle (also denoted by $H$) on $\T\llbracket t \rrbracket$ with coefficients in the conformal bimodule $\U\llbracket t \rrbracket$.
	\begin{defn}Let $\R : \U \to \T[\l]$ be a conformal $H$-TRB operator. A formal one-parameter deformation of $\R$  is given by $\R_t = \sum_{i\geq 0}t_i\R_i \in Hom(\U, \T)\llbracket t \rrbracket$  which is a  formal sum with $\R_0 = \R$  in such a way that $\R_t : \U\llbracket t \rrbracket \to \T\llbracket t \rrbracket[\l]$ is a conforfmal $H$-TRB operator. Alternatively, for any $u_1, u_2 \in \U $, we have   the following  equation: $$\R_t(u_1)_\l\R_t(u_2) = \R_t({u_1}_\l \R_t(u_2) + \R_t(u_1) _\l u_2 + H_\l(\R_t(u_1), \R_t(u_2)).$$
	\end{defn}Thus, the following system of equations is valid in a formal deformation,  for $n\geq 0$, $$\sum_{i+j= n}\R_i(u_1)_\l\R_j (u_2) =\sum_{i+j= n}\R_i({u_1} _\l \R_j (u_2) + \R_j (u_1) _\l u_2) + \sum_{i+j+k= n} \R_i(H_\l(\R_j(u_1), \R_k(u_2))).$$The equations mentioned above are called deformation equations. Consider the following cases:
 \begin{enumerate}
     \item When  $n = 0$,  then $\R_0 = \R$ indicating that $\R$ is an $H$-TRB operator on the associative conformal algebra $\T$. 
     \item By using the same approach as in Eq. (\ref{eq30}), we  determine that when  $n = 1$, then  $\R_1$ corresponds to a $1$-cocycle within the cohomology of $\R$. This is referred to as the "infinitesimal of the deformation".
 \end{enumerate}
\begin{defn}
		Twoone-parameter formal deformations of a conformal $H$-TRB operator $\R$ denoted by $\R_t = \sum_{i\geq 0}t^i\R_i$ and $\R'_t = \sum_{i\geq 0}t^i\R'_i$ are considered  to be equivalent if there exist$\phi_i \in Hom(\T, \T)$ and $\psi_i \in Hom(\U, \U),$  for $i\geq 2$ being  two linear maps and an element $p \in \T$, such that the following equation holds for all values of $t$
		\begin{equation}
		(\phi_t= id_\l+ t(ad_p^l- ad_p^r)+ \sum_{i\geq 2}t^i\phi_i, \psi_t= id_\U+ t(\mathfrak{l}_p- \mathfrak{r}_p+ H_\l(p,\R-)- H_\l(\R-,p))+\sum_{i\geq 2}t^i\psi_i).
		\end{equation}This equation represents a conformal $H$-TRB operators'  morphism $\R_t\to \R'_t$.
	\end{defn}Similar to linear deformations, we obtain the following equation: $$\R_1(u)-\R'_1 (u) = d_\R(p)_\l(u)$$ for $u \in \U$. Thus, we derive the following result.
	\begin{thm}
		Let $\R_t =\sum_{i\geq 0}t^i\R_i$ represents a one-parameter formal deformation of a conformal $H$-TRB operator $\R$. In this case, the linear term $\R_1$ is a $1$-cocycle in the cohomology of conformal $H$-TRB operator $\R$. And its cohomology class is only determined by the equivalence class of the deformation $\R_t.$
	\end{thm}
	In the subsequent theorem, we show the constraint when  the deformed  conformal $H$-TRB operator $\R_t$ becomes equivalent to the  undeformed $H$-TRB operator $\R$.
	\begin{thm}
		Consider a conformal $H$-TRB operator  $\R$ that fulfills $Z^1_\R (\U,\T) = d_\R (Nij(\R))$. In this case, any formal deformation $\R_t$ of conformal $H$-TRB operator $\R$ is equal to the conformal $H$-TRB operator $\R$ that is not deformed, i.e., $\R'_t = \R$.
	\end{thm}
	\begin{proof} Assuming that any formal deformation of $\R$ is represented by $\R_t = \sum_{i\geq 0}t^i\R_i$ . According to the previous theorem, we conclude that the linear term $\R_1\in Z^1_{\R}(\U,\T)$. Thus based on the assumption, there is an element $p\in Nij(\R)$ that satisfies $\R_1 = d_\R(p)$. We  define it as follows
		\begin{equation}
		(\phi_t=id_\l+t(ad_p^l-ad_p^r), \psi_t=id_\U+ t(\mathfrak{l}_p- \mathfrak{r}_p+ H_\l(p,\R-)- H_\l(\R-,p)))
		\end{equation}
		and assume $\R'_t:= \phi_t \R_t  \psi^{-1}_t$. Then by our assumption, $\R_t$ is equivalent to $\R'_t$. Moreover, we get
		\begin{equation*}
		\begin{aligned}
		&\R'_t(u):= \phi_t  \R_t  \psi^{-1}_t\\&
		=(id_\l+t(ad_p^l-ad_p^r))(\R+t\R_1)(id_\U- t(\mathfrak{l}_p- \mathfrak{r}_p+ H_\l(p,\R-)- H_\l(\R-,p)))(u)\\&= (id_\l+t(ad_p^l-ad_p^r))(\R(u)+t\R_1(u)- t(\R(p_\l u- u_{-\p-\l}p+ H_\l(p,\R (u))- H_\l(\R (u),p)))\\&-t^2\R_1(p_\l u- u_{-\p-\l}p+ H_\l(p, \R (u))- H_\l(\R (u),p)))\text{  (mod  $t^2$) }\\&=(id_\T+t(ad_p^l-ad_p^r))(\R(u)+t(\R_1(u)-\R(p_\l u- u_{-\p-\l}p+ H_\l(p,\R( u))- H_\l(\R (u),p)))) \text { (mod  $t^2$) }\\&=\R(u)+t(\R_1(u)-\R(p_\l u- u_{-\p-\l}p+ H_\l(p,\R( u))- H_\l(\R (u),p)))+t(ad_p^l-ad_p^r)\R(u) \text{ (mod  $t^2$) }\\&=\R(u)+t(\R_1(u)-\R(p_\l u- u_{-\p-\l}p+ H_\l(p,\R (u))- H_\l(\R (u),p)))+ t(p_\l\R(u)- \R(u)_{-\p-\l}p)\\&= \R(u)+t(\R_1(u)-\R(p_\l u- u_{-\p-\l}p+ H_\l(p,\R (u))- H_\l(\R (u),p))+p_\l\R(u)-\R(u)_{-\p-\l}p)
		\\&=\R(u)
		\end{aligned}
		\end{equation*} Last equality is acquired by employing the the fact that $\R_1 = d_\R(p)$.     This means that the coefficient of $t$ in $\R'_t$  becomes equal to zero. By involving the identical approach for $mod~t^n$ for $n\geq 3$, one establishes the equivalence between the deformations $\R$  and $\R_t$. It concludes the proof.\end{proof}
	\begin{rem}A conformal $H$-TRB operator $\R$ on an associative conformal algebra is called as rigid if any formal deformation of $\R$ can be transformed into the original undeformed state. Hence, the earlier theorem states that rigidity  of $\R$ can be sufficiently ensured by the condition $Z^{1}_{\R}(\U,\T)= d_{\R}(Nij(\R))$.
	\end{rem}
	
\end{document}